\documentclass[a4paper,11pt]{article}

\usepackage{amsmath}
\usepackage{amssymb}
\usepackage{amsthm}
\usepackage{amsfonts}
\usepackage{mathtools}
\usepackage{ascmac}
\usepackage{spalign}
\usepackage{fancybox} 
\usepackage[mathscr]{eucal}
\usepackage{footnpag}
\usepackage{siunitx}
\usepackage{multicol}
\usepackage{bbm}
\usepackage{framed}
\usepackage{here}

\usepackage{graphicx}
\usepackage{xcolor}
\usepackage{wrapfig}
\usepackage{float}

\theoremstyle{definition}
\newtheorem{theorem}{Theorem}[section]
\newtheorem*{theorem*}{Theorem}

\newtheorem*{definition*}{Definition}
\newtheorem{proposition}{Proposition}[section]
\newtheorem*{proposition*}{Proposition}
\newtheorem{lemma}{Lemma}[section]
\newtheorem*{lemma*}{Lemma}
\newtheorem{corollary}{Corollary}[section]
\newtheorem*{corollary*}{Corollary}
\newtheorem{remark}{Remark}[section]
\newtheorem*{remark*}{Remark}

\newfont{\bg}{cmr9 scaled\magstep4}
\newcommand{\bigzero}{
\smash{\lower1.0ex\hbox{\bg 0}}}

\newcommand{\ctext}[1]{\raise0.2ex\hbox{\textcircled{\scriptsize{#1}}}}

\title{Effective resistance and spanning trees \\in complete graphs with distance-class deletions}

\author{Shunya Tamura\thanks{Okegawa City Okegawa West Junior High School, Saitama, 363-0027, Japan, e-mail: shunya.tamura059@gmail.com}
}

\date{}
\begin{document}
\maketitle

%%%%%%%%%%%%%%%%%%%%%%%%%%%%%%%%%%%%%%%%
\begin{abstract}
In this paper, we consider circulant graphs obtained from the complete graph $K_N$ by deleting all edges belonging to a prescribed distance class.  
We study, in a unified manner, the effective resistance, the expected hitting time, the number of spanning trees, and the number of two-component spanning forests of these graphs.

For general distance-class deletions, these quantities admit natural spectral representations in terms of the Laplacian eigenvalues.  
However, such representations typically remain at the level of finite Fourier sums, and concise closed forms are not expected in general.
We focus on the case of a single deleted distance class.
When the number of vertices $N$ is odd and $\gcd(r,N)=1$, the graph $G_{N,r}$ is isomorphic to $G_{N,1}$.

In this setting, we derive explicit exponential-type formulas for the effective resistance and the number of spanning trees, and obtain corresponding closed expressions for two-component spanning forests and expected hitting times.
Our results show that the case $r=2$ is not essentially new, but follows from a general isomorphism structure underlying distance-class deletions.
We also clarify the relation of our formulas to earlier results on the complete graph with a Hamiltonian cycle removed, and provide a unified derivation within a spectral framework.
Moreover, by asymptotic analysis, we show that the ratio $\tau(G_{N,1})/\tau(K_N)$ converges to $e^{-2}$ as $N \to \infty$.
\end{abstract}



%%%%%%%%%%%%%%%%%%%%%%%%%%%%%%%%%%%%%%%%
%%%%%%%%%%%%%%%%%%%%%%%%%%%%%%%%%%%%%%%%
\section{Introduction} \label{sec01}

The effective resistance on a graph is a fundamental and powerful concept that connects electrical network theory, spectral graph theory, random walk theory, and the enumeration of spanning trees \cite{Ref01, Ref03, Ref05, Ref06, Ref07}.  
By the classical commute time identity, the effective resistance is closely related to the expected hitting time of a simple random walk.  
Through Kirchhoff’s matrix--tree theorem, it is also deeply connected with the number of spanning trees and the number of two-component spanning forests \cite{Ref03, Ref08, Ref02}.  
For highly symmetric graphs, these quantities admit spectral representations in terms of Laplacian eigenvalues.  
However, examples where these representations lead to explicit and readable closed forms are rather limited \cite{Ref07, Ref09, Ref10}. 

In this paper, we consider circulant graphs obtained from the complete graph $K_N$ by deleting all edges belonging to a prescribed distance class.  
Identifying the vertex set with $\mathbb Z_N$, we denote by $G_{N,r}$ the graph obtained by removing all edges of circulant distance $\pm r$.  
These graphs form a natural family that preserves vertex transitivity while modifying only the local structure.  
They may be regarded as symmetric perturbations of the complete graph.  
In other words, we study the following question:  
how do spectral and combinatorial quantities change when one performs a minimal local modification of a graph with extremely high symmetry?

For general distance-class deletions, the Laplacian eigenvalues can be written explicitly in trigonometric form, and the effective resistance, the expected hitting time, the number of spanning trees, and the number of two-component spanning forests are naturally expressed as spectral sums.  
However, these expressions usually remain Fourier sums, and concise closed forms are not expected without additional arithmetic structure.  
In this paper, we focus on the simplest nontrivial case of a single deleted distance class.  
When $\gcd(r,N)=1$, this case is isomorphic to the case $r=1$; the cases $\gcd(r,N)>1$ and general multiple deleted distance classes are not treated here.

A closely related problem has been studied by Chair \cite{Ref09}, who derived exact expressions for two-point resistances and random-walk quantities on the complete graph with a Hamiltonian cycle removed.  
Since $G_{N,1}$ is precisely the complete graph with the edges of a Hamiltonian cycle removed, the resistance formula in Theorem \ref{thm41} is closely related to the results of \cite{Ref09}.  
Thus, the resistance formula itself should be viewed as a reformulation and streamlined derivation of known formulas within a unified spectral framework.  
The main contribution of the present paper is to connect this resistance formula systematically with spanning trees, two-component spanning forests, hitting times, and asymptotic enumeration.

Our main result shows that when $r=1$ and the number of vertices $N$ is odd, these quantities admit remarkably simple closed forms.

\medskip
\noindent
\textbf{Main Theorem A.}
Let $N \ge 5$ be odd, and let $G_{N,1}$ be the graph obtained from the complete graph $K_N$ by deleting all edges of circulant distance $\pm1$.  
For $u,v\in\mathbb Z_N$, let
\[
q=q(u,v)\equiv v-u \pmod N,\qquad q\in\{0,1,\dots,N-1\}.
\]
Put
\[
\Delta=\sqrt{N(N-4)}, 
\qquad
\rho=\frac{N-2+\Delta}{2}.
\]
Then the effective resistance is given by
\[
R^{(1)}(u,v)
=\frac{2}{\Delta(\rho^N+1)}
\left\{
\rho^N-1+(-1)^q(\rho^q-\rho^{N-q})
\right\}.
\]

The right-hand side is invariant under $q\mapsto N-q$; hence the resistance depends only on the unoriented circulant distance between $u$ and $v$
Moreover, the number of spanning trees admits the closed form
\[
\tau(G_{N,1})
=\frac{1}{N}\frac{(\rho^N+1)^2}{\rho^{N-1}(\rho+1)^2}.
\]
Using the matrix--tree theorem, the all-minors matrix--tree theorem, and the commute time identity, we also obtain explicit formulas for two-component spanning forests and expected hitting times.

The essential reason for this simplicity is that the cosine term appearing in the denominator of the eigenvalues reduces, via a complex-variable representation, to a quadratic polynomial.  
Its two roots $\rho$ and $\rho^{-1}$ allow a complete evaluation of the Fourier sums and lead to exponential-type closed expressions.

The case $r=2$ is not essentially different when $N$ is odd.  
More generally, whenever $\gcd(r,N)=1$, multiplication by $r^{-1}$ gives an automorphism of the cyclic group $\mathbb Z_N$, and this induces a graph isomorphism between $G_{N,r}$ and $G_{N,1}$.

\medskip
\noindent
\textbf{Main Theorem B.}
Let $N$ be odd and let $r$ be an integer such that $\gcd(r,N)=1$.  
Define
\[
\varphi_r(x)=r^{-1}x \pmod N.
\]
Then $\varphi_r$ defines a graph isomorphism
\[
G_{N,r}\cong G_{N,1}.
\]
Consequently,
\[
\tau(G_{N,r})=\tau(G_{N,1}),
\]
and
\[
R_{G_{N,r}}(u,v)
=
R_{G_{N,1}}(\varphi_r(u),\varphi_r(v)).
\]
In particular, the case $r=2$ is obtained as a special case.

We also derive asymptotic consequences.  
For each fixed distance independent of $N$, as $N\to\infty$ through odd integers, the resistance distance has the same order $2/N$ as that of the complete graph.  
On the other hand, the number of spanning trees satisfies
\[
\lim_{N\to\infty}
\frac{\tau(G_{N,1})}{\tau(K_N)}
=
e^{-2}.
\]
These consequences are not explicitly derived in \cite{Ref09} and arise naturally from the spectral approach developed here.

This paper is organized as follows.  
In Section \ref{sec02}, we review the basic setting of circulant graphs with distance-dependent weights and summarize known results on Laplacian eigenvalues, effective resistance, the matrix--tree theorem, and the commute time identity.  
In Section \ref{sec03}, we derive spectral representations for connected general distance-class deletion graphs and clarify the role of the deleted distance set. 

In Section \ref{sec04}, we first treat $r=1$ explicitly and then reduce all coprime single-distance deletions, including $r=2$, to this case by a graph isomorphism.
In Section \ref{sec05}, we analyze the asymptotic behavior and structural consequences of these formulas.  
Finally, Section \ref{sec06} contains concluding remarks and possible directions for further study.

%%%%%%%%%%%%%%%%%%%%%%%%%%%%%%%%%%%%%%%%
%%%%%%%%%%%%%%%%%%%%%%%%%%%%%%%%%%%%%%%%
\section{Preliminaries and Known Results} \label{sec02}

Let $N \ge 3$ be an integer, and identify the vertex set with the cyclic group
\[
V=\mathbb Z_N=\{0,1,\dots,N-1\}.
\]

For any $u,v \in V$, we define the oriented difference
\[
q(u,v) := v - u \pmod N, \quad q \in \{0,1,\dots,N-1\}.
\]
We also define the (unoriented) circulant distance by
\[
h(u,v) := \min\{q(u,v),\, N - q(u,v)\}.
\]
Note that $q(u,v)$ and $h(u,v)$ play different roles: spectral formulas are naturally expressed in terms of the oriented residue $q$, while graph-theoretic quantities such as resistance ultimately depend only on the unoriented distance $h$.

Let
\[
w:\{1,2,\dots,\lfloor N/2\rfloor\}\to \mathbb R_{\ge 0}
\]
be a distance-dependent weight function.  
We denote by $G_{N,w}$ the weighted undirected graph obtained by joining any two vertices at distance $k$ by an edge of weight $w(k)$.  

Throughout this paper, we assume that $G_{N,w}$ is connected unless otherwise stated.  
Equivalently, all nonzero Laplacian eigenvalues satisfy $\lambda_j>0$ for $1\le j\le N-1$.  
All formulas for effective resistance, spanning trees, spanning forests, and hitting times are used under this connectivity assumption.

In particular:
\begin{itemize}
\item
If $w(k)\equiv 1$, then $G_{N,w}$ coincides with the complete graph $K_N$.

\item
If $w(r)=0$ for some distance $r$, then $G_{N,w}$ is the circulant graph obtained from $K_N$ by deleting all edges of distance $r$.
\end{itemize}

Let $A$ be the weighted adjacency matrix of $G_{N,w}$ and $D$ the diagonal degree matrix.  
We define the Laplacian matrix by $L:=D-A$.  
Since $G_{N,w}$ is vertex-transitive, the weighted degree of each vertex is given by
\[
\deg
=\sum_{k=1}^{\left\lfloor \frac{N-1}{2} \right\rfloor} 2w(k)
+\mathbf{1}_{\{N\text{ even}\}}\, w\!\left(\frac{N}{2}\right),
\]
where $\mathbf{1}_{\{N\text{ even}\}}$ denotes the indicator function which equals $1$ if $N$ is even and $0$ otherwise.

It is well known that circulant matrices are diagonalized by the discrete Fourier basis \cite{Ref04}.  

%%%%%%%%%%%%%%%%%%%%%%%%%%%%%%%%%%%%%%%%
\begin{proposition} \label{prop21}
For $j=0,1,\dots,N-1$, the eigenvalues of the Laplacian matrix $L$ are given by
\[
\lambda_j
=\sum_{k=1}^{\left\lfloor \frac{N-1}{2} \right\rfloor}
2w(k)\left(1-\cos\frac{2\pi jk}{N}\right)
+\mathbf{1}_{\{N\text{ even}\}}
w\left(\frac{N}{2}\right)\left(1-\cos(\pi j)\right).
\]
\end{proposition}

In particular, $\lambda_0=0$, and $G_{N,w}$ is connected if and only if $\lambda_j>0$ for $1\le j\le N-1$.

For circulant graphs, since $L$ is diagonalized by the Fourier basis, the effective resistance admits the following spectral representation.

%%%%%%%%%%%%%%%%%%%%%%%%%%%%%%%%%%%%%%%%
\begin{proposition} \label{prop22}
\begin{align}
R(u,v)
=\frac{2}{N}
\sum_{j=1}^{N-1}
\frac{1-\cos\left(2\pi j q(u,v)/N\right)}{\lambda_j},
\label{eq:resistance}
\end{align}
where $\lambda_j$ is given in Proposition \ref{prop21}.
\end{proposition}

This formula is naturally expressed in terms of the oriented residue $q(u,v)$; however, since the cosine function is even, the resulting resistance depends only on the unoriented distance $h(u,v)$.

Such formulas appear in the literature as resistance formulas for circulant networks \cite{Ref07}.

Let $\tau(G_{N,w})$ denote the number of weighted spanning trees of $G_{N,w}$.  
By the weighted matrix--tree theorem, we have the following.

%%%%%%%%%%%%%%%%%%%%%%%%%%%%%%%%%%%%%%%%
\begin{theorem}[G. Kirchhoff, \cite{Ref08}] \label{thm21}
\[
\tau(G_{N,w})
=\frac{1}{N}\prod_{j=1}^{N-1}\lambda_j.
\label{eq:spanningtree}
\]
\end{theorem}

For proofs, see \cite{Ref02, Ref01, Ref08}.

Let $F(u \mid v)$ denote the number of two-component spanning forests in which $u$ and $v$ belong to distinct connected components.  

%%%%%%%%%%%%%%%%%%%%%%%%%%%%%%%%%%%%%%%%
\begin{theorem}[S. Chaiken, \cite{Ref02}] \label{thm22}
\begin{align}
F(u \mid v)=\tau(G_{N,w}) \cdot R(u,v).
\label{eq:forest}
\end{align}
\end{theorem}

Theorem \ref{thm22} follows from the all-minors matrix--tree theorem or the matrix--forest theorem \cite{Ref02,Ref01}.  
In this paper, we use Theorem \ref{thm22} as a fundamental bridge in deriving closed forms.

Next, consider the simple random walk on $G_{N,w}$, where transition probabilities are normalized by weighted degrees.  
Let $H(u,v)$ denote the expected hitting time from $u$ to $v$.  

%%%%%%%%%%%%%%%%%%%%%%%%%%%%%%%%%%%%%%%%
\begin{theorem}[Chandra et al. \cite{Ref03, Ref05}] \label{thm23}
For the weighted random walk,
\[
H(u,v)+H(v,u)=\mathrm{vol}(G)\cdot R(u,v),
\]
where $\mathrm{vol}(G)=\sum_{x\in V}\deg(x)$ denotes the total weighted degree.
\end{theorem}

In the present setting, the symmetry $H(u,v)=H(v,u)$ follows from an explicit graph automorphism.  
Indeed, for any $u,v\in\mathbb Z_N$, the map
\[
x \mapsto u+v-x \pmod N
\]
is an automorphism of $G_{N,w}$ that interchanges $u$ and $v$.  
Hence the distributions of hitting times from $u$ to $v$ and from $v$ to $u$ coincide, and therefore $H(u,v)=H(v,u)$.

Substituting this symmetry into the commute time identity yields
\begin{equation}
H(u,v)=\frac{\mathrm{vol}(G)}{2}\,R(u,v).
\label{formula01}
\end{equation}

For general distance-dependent weights $w$ or general distance-class deletions, 
the spectral expressions in Propositions \ref{prop21} and \ref{prop22} do not usually simplify further.  
In particular, concise closed forms are not expected without additional structure.

In the following sections, we focus on the case of a single deleted distance class.  
When $\gcd(r,N)=1$, this case reduces to $r=1$ via a graph isomorphism, and we derive complete closed formulas for the effective resistance, the numbers of spanning trees and spanning forests, and the expected hitting times.

%%%%%%%%%%%%%%%%%%%%%%%%%%%%%%%%%%%%%%%%%%%%%%%%%%%%%%%%%%%%%%%%%%%%%%%%%%%%%%
%%%%%%%%%%%%%%%%%%%%%%%%%%%%%%%%%%%%%%%%%%%%%%%%%%%%%%%%%%%%%%%%%%%%%%%%%%%%%%
\section{Spectral representations for general distance-class deletions} \label{sec03}

In this section, we show that for circulant graphs with general distance-class deletions and distance-dependent weights, the effective resistance, expected hitting times, spanning tree counts, and two-component spanning forest counts are described in a unified manner by the same spectral data.  
Although the discussion does not yield closed forms in general, it provides the basis for the explicit derivations in the next section.

Let $S \subset \{1,2,\dots,\lfloor N/2\rfloor\}$ be a set of distances.  
We write $G_{N,S}$ for the circulant graph obtained from the complete graph $K_N$ by deleting all edges whose circulant distance belongs to $S$.

In the following, we assume that $G_{N,S}$ is connected.
Equivalently, the Laplacian eigenvalues satisfy
\[
\lambda_j > 0 \quad (1 \le j \le N-1).
\]
All formulas for resistance, spanning trees, spanning forests, and hitting times are used under this connectivity assumption.

Then $G_{N,S}$ can be regarded as a weighted graph $G_{N,w_S}$ with the distance-dependent weight
\[
w_S(k)=
\begin{cases}
0 & (k\in S),\\
1 & (k\notin S).
\end{cases}
\]
By Proposition \ref{prop21}, the Laplacian eigenvalues $\lambda_j$ of $G_{N,S}$ are given by
\begin{align*}
\lambda_j
&=
\sum_{\substack{1\le k\le \left\lfloor\frac{N-1}{2}\right\rfloor\\ k\notin S}}
2\left(1-\cos\frac{2\pi jk}{N}\right)
+\mathbf{1}_{\{N\text{ even}\}}\mathbf{1}_{\{N/2\notin S\}}
\left(1-\cos(\pi j)\right).
\end{align*}

On the other hand, for the complete graph $K_N$, the Laplacian eigenvalues are given by
\[
\lambda_j = N \quad (j \ne 0).
\]
Therefore, when deleting distance classes $S$, the contribution removed from the Laplacian
is not simply a cosine term, but rather
\[
2\left(1 - \cos\frac{2\pi jk}{N}\right).
\]
Hence, for odd $N$ and $S \subseteq \{1,2,\dots,(N-1)/2\}$, we obtain
\[
\lambda_j
=\sum_{\substack{1 \leq k \le (N-1)/2 \\ k \notin S}}
2\left(1 - \cos\frac{2\pi jk}{N}\right)
=N - \sum_{k \in S}2\left(1 - \cos\frac{2\pi jk}{N}\right), \quad 1 \leq j \leq N-1.
\]
Equivalently,
\[
\lambda_j
=N-2|S|+2 \sum_{k \in S} \cos\left(\frac{2\pi jk}{N}\right), \quad 1 \leq j \leq N-1.
\]

This shows that the effect of deleting distance classes is directly reflected in the Laplacian eigenvalues as a finite sum of cosine terms.  
In particular, as $|S|$ increases, the structure of $\lambda_j$ becomes a trigonometric polynomial of higher complexity.

This expression is consistent with the identity
\[
\sum_{k=1}^{(N-1)/2}
2\left(1 - \cos\frac{2\pi jk}{N}\right) = N,
\quad j \ne 0,
\]
used in Lemma \ref{lem41}.

Once the Laplacian eigenvalues $\{\lambda_j\}_{j=1}^{N-1}$ are given, the following quantities are determined solely by this spectral data:
\begin{itemize}
\item
the effective resistance
\[
R(u,v)
=\frac{2}{N}\sum_{j=1}^{N-1}\frac{1-\cos\left(2\pi j\,q(u,v)/N\right)}{\lambda_j},
\]
\item
the number of spanning trees
\[
\tau(G_{N,S})
=\frac{1}{N}\prod_{j=1}^{N-1}\lambda_j,
\]
\item
the number of two-component spanning forests
\[
F(u\mid v)=\tau(G_{N,S}) \cdot R(u,v),
\]
\item
and, by the symmetry argument in Section \ref{sec02}, the expected hitting time

\[
H(u,v)=\frac{\mathrm{vol}(G_{N,S})}{2}\,R(u,v).
\]
\end{itemize}

In this sense, for distance-class deletion graphs, the effective resistance, the spanning tree count, the spanning forest count, and the expected hitting time are described in a unified way by the same eigenvalue information.

As Proposition \ref{prop21} indicates, for a general distance set $S$, the Laplacian eigenvalues $\lambda_j$ are given as a linear combination of finitely many cosine terms.  
Since the effective resistance $R(u,v)$ is given by Proposition \ref{prop22}, Fourier analysis shows that $R(u,v)$ can be expressed as a linear combination of finitely many exponential terms.  
The number of such terms increases with $|S|$, and therefore concise closed forms are not expected in general.

The structure of the graph depends not only on the number of deleted distance classes, but also on their arithmetic properties.
In particular, the following distinctions are important:
\begin{itemize}
\item[(i)] Single deleted classes with $\gcd(r,N)=1$, which reduce to the case $r=1$ by a graph isomorphism;
\item[(ii)] Single deleted classes with $\gcd(r,N)>1$, which may exhibit different structural behavior and are not analyzed in this paper;
\item[(iii)] Multiple deleted classes $S$, considered up to multiplication by units modulo $N$;
\item[(iv)] The special even-$N$ class $N/2$, which contributes only one edge per vertex.
\end{itemize}

In this paper, we focus on the simplest case of a single deleted distance class with $\gcd(r,N)=1$, and do not attempt to treat the other cases.

In the next section, we focus on the case $r=1$, and show that in this situation the spectral expressions admit complete closed forms.  
The case $r=2$ will then be treated as a special instance of the general isomorphism described above.

We emphasize that the case gcd$(r,N)>1$ may lead to fundamentally different spectral structures,
and is left as an interesting direction for future work.
%%%%%%%%%%%%%%%%%%%%%%%%%%%%%%%%%%%%%%%%%%%%%%%%%%%%%%%%%%%%%%%%%%%%%%%%%%%%%%
%%%%%%%%%%%%%%%%%%%%%%%%%%%%%%%%%%%%%%%%%%%%%%%%%%%%%%%%%%%%%%%%%%%%%%%%%%%%%%
\section{Closed forms for distance-class deletions} \label{sec04}

Throughout this section, we assume that $N \ge 5$ is odd.  
In addition, all graphs considered in this section are assumed to be connected,
so that the spectral formulas for resistance, spanning trees,
spanning forests, and hitting times are valid.
Under this assumption, sign issues disappear and all resulting closed forms are naturally expressed as positive quantities.

%%%%%%%%%%%%%%%%%%%%%%%%%%%%%%%%%%%%%%%%%%%%%%%%%%%%%%%%%%%%%%%%%%%%%%%%%%%%%%
\subsection{The complete graph with distance-$1$ edges removed} \label{sub41}
%%%%%%%%%%%%%%%%%%%%%%%%%%%%%%%%%%%%%%%%%%%%%%%%%%%%%%%%%%%%%%%%%%%%%%%%%%%%%%

Let $G_{N,1}$ be the circulant graph obtained from the complete graph $K_N$ by deleting all edges of distance $1$ (i.e., $\pm1$).  
Equivalently, the vertex set is $\mathbb Z_N$, and each $i\in\mathbb Z_N$ is adjacent to every vertex except $i\pm 1$.  
Thus each vertex has degree $N-3$.

For $u,v \in \mathbb Z_N$, we use the notation
\[
q=q(u,v)\equiv v-u \pmod N,
\]
and
\[
h(u,v)=\min\{q,\,N-q\},
\]
as introduced in Section \ref{sec02}.

\medskip
We use the following constants throughout:
\begin{align}
\Delta:=\sqrt{N(N-4)},\qquad
\rho:=\frac{N-2+\Delta}{2}>1.
\label{eq:rhoDelta}
\end{align}
Then
\begin{equation}
\rho+\rho^{-1}=N-2,\qquad
\rho-\rho^{-1}=\Delta.
\label{eq:rho-id}
\end{equation}

%%%%%%%%%%%%%%%%%%%%%%%%%%%%%%%%%%%%%%%%
\begin{lemma} \label{lem41}
The Laplacian eigenvalues of $G_{N,1}$ are given by
\[
\lambda_0=0,\qquad
\lambda_j=(N-2)+2\cos\left(\frac{2\pi j}{N}\right)
\quad (j=1,2,\dots,N-1).
\]
\end{lemma}

\begin{proof}
Since only the distance-$1$ edges are removed, all edges of distances $k\ge 2$ are present.  
In Proposition \ref{prop21}, we put $w(1)=0$ and $w(k)=1$ for $k\ge 2$.  
Then for $j\neq 0$,
\[
\lambda_j
=
\sum_{k=2}^{(N-1)/2}
2\left(1-\cos\frac{2\pi jk}{N}\right).
\]
On the other hand, the graph having all distance-$k$ edges for $k=1,2,\dots,(N-1)/2$ is the complete graph $K_N$, and hence for $j\neq 0$,
\[
\sum_{k=1}^{(N-1)/2}2\left(1-\cos\frac{2\pi jk}{N}\right)=N.
\]
Therefore,
\begin{align*}
\lambda_j
&=\sum_{k=1}^{(N-1)/2}2\left(1-\cos\frac{2\pi jk}{N}\right)-2\left(1-\cos\frac{2\pi j}{N}\right) \\
&=N-2+2\cos\left(\frac{2\pi j}{N}\right),
\end{align*}
as claimed.

This expression is consistent with the identity
\[
\sum_{k=1}^{(N-1)/2}2\left(1-\cos\frac{2\pi jk}{N}\right)=N,
\quad j\neq 0,
\]
used in Section \ref{sec03}.

\end{proof}

%%%%%%%%%%%%%%%%%%%%%%%%%%%%%%%%%%%%%%%%
\begin{theorem} \label{thm41}
Let $N \ge 5$ be odd, and let $G_{N,1}$ be the graph obtained from $K_N$ by deleting all edges of distance $1$.  
Identify the vertex set with $\mathbb Z_N=\{0,1,\dots,N-1\}$, and for $u,v\in\mathbb Z_N$ let
\[
q=q(u,v)\equiv v-u \pmod N,\qquad q\in\{0,1,\dots,N-1\}.
\]
Then the effective resistance in $G_{N,1}$ is given by
\[
R^{(1)}(u,v)
=\frac{2}{\Delta(\rho^N+1)}
\left\{
\rho^{N}-1+(-1)^{q}(\rho^{q}-\rho^{N-q})\right\}.
\]
Moreover, the right-hand side is invariant under $q\mapsto N-q$.
Hence the resistance depends only on the unoriented circulant distance
\[
h(u,v)=\min\{q,N-q\}.
\]
If $u\neq v$, then $R^{(1)}(u,v)>0$.
\end{theorem}

\begin{proof}
By Proposition \ref{prop22} and Lemma \ref{lem41}, for $q\equiv v-u \pmod N$ we obtain
\begin{equation} \label{eq:Rfourier-correct-rev}
R^{(1)}(u,v)
=\frac{1}{N}\sum_{j=1}^{N-1}
\frac{2\left(1-\cos(2\pi jq/N)\right)}
{(N-2)+2\cos(2\pi j/N)}.
\end{equation}

If $q=0$, then $u=v$ and hence $R^{(1)}(u,v)=0$. 
The stated closed form also gives $0$ in this case. 
Thus, in the following, we assume $1\le q\le N-1$.

Let $z:=\omega^j$. Then
\[
\cos\!\left(\frac{2\pi j}{N}\right)=\frac12(z+z^{-1}),\qquad
\cos\!\left(\frac{2\pi jq}{N}\right)=\frac12(z^{q}+z^{-q}),
\]
and hence
\[
(N-2)+2\cos\!\left(\frac{2\pi j}{N}\right)
=\frac{z^2+(N-2)z+1}{z},
\qquad
2\left(1-\cos\!\left(\frac{2\pi jq}{N}\right)\right)
=2-(z^{q}+z^{-q}).
\]
Therefore,
\begin{align}
\frac{2\left(1-\cos(2\pi jq/N)\right)}{(N-2)+2\cos(2\pi j/N)}
&=\frac{z\left(2-(z^q+z^{-q})\right)}{z^2+(N-2)z+1}  \notag\\
&=\frac{2z-(z^{q+1}+z^{1-q})}{z^2+(N-2)z+1}.
\label{eq:term-z-rev}
\end{align}
Since $z^N=1$, we may rewrite
\[
z^{1-q}=z^{N-q+1}.
\]

Since $N\ge 5$, the discriminant of $z^2+(N-2)z+1$ is
$(N-2)^2-4=N(N-4)>0$.  

Let $\rho>1$ be the real number satisfying $\rho+\rho^{-1}=N-2$.  
Then
\[
z^2+(N-2)z+1=(z+\rho)(z+\rho^{-1}).
\]
Putting $\Delta:=\rho-\rho^{-1}$, we have the partial fraction decomposition
\begin{equation}\label{eq:pfrac-rev}
\frac{1}{(z+\rho)(z+\rho^{-1})}
=\frac{1}{\Delta}\left(\frac{1}{z+\rho^{-1}}-\frac{1}{z+\rho}\right).
\end{equation}

We next use the following finite root-of-unity sum. 
Let $N$ be odd, $\omega=e^{2\pi i/N}$, and $\rho>0$. 
For any integer $m$, let $\overline m$ be the representative of $m$ modulo $N$ in $\{0,1,\dots,N-1\}$, and define
\[
S_m(\rho):=\sum_{j=0}^{N-1}\frac{\omega^{jm}}{\rho+\omega^j}.
\]
Since $\omega^N=1$, we have $S_m(\rho)=S_{\overline m}(\rho)$. Moreover,
\[
S_m(\rho)=S_{\overline m}(\rho)
=
\begin{cases}
\displaystyle \frac{N\rho^{N-1}}{\rho^N+1}, & \overline m=0, \\[2mm]
\displaystyle -\frac{N(-1)^{\overline m}\rho^{\overline m-1}}{\rho^N+1}, & 1\le \overline m\le N-1.
\end{cases}
\]
This identity follows from the factorization
\[
x^N-1=\prod_{j=0}^{N-1}(x-\omega^j)
\]
together with the orthogonality relation
\[
\sum_{j=0}^{N-1}\omega^{j\ell}
=
\begin{cases}
N, & \ell\equiv 0 \pmod N,\\
0, & \ell\not\equiv 0 \pmod N.
\end{cases}
\]
The corresponding identity with denominator $\rho^{-1}+\omega^j$ is obtained by replacing $\rho$ with $\rho^{-1}$.

From \eqref{eq:Rfourier-correct-rev} and \eqref{eq:term-z-rev}, we have
\begin{equation}\label{eq:Rz-rev}
R^{(1)}(u,v)
=\frac{1}{N}\sum_{j=1}^{N-1}
\frac{2z-(z^{q+1}+z^{N-q+1})}{(z+\rho)(z+\rho^{-1})}.
\end{equation}
At $j=0$ (i.e., $z=1$), the numerator becomes $2-(1+1)=0$, so we may include the $j=0$ term:
\[
\sum_{j=1}^{N-1}
\frac{2z-(z^{q+1}+z^{N-q+1})}{(z+\rho)(z+\rho^{-1})}
=
\sum_{j=0}^{N-1}
\frac{2z-(z^{q+1}+z^{N-q+1})}{(z+\rho)(z+\rho^{-1})}.
\]
Substituting \eqref{eq:pfrac-rev} into \eqref{eq:Rz-rev}, we obtain
\begin{equation}\label{eq:Rsplit-rev}
R^{(1)}(u,v)
=\frac{1}{N\Delta}\sum_{j=0}^{N-1}
\left\{
\frac{2z-(z^{q+1}+z^{N-q+1})}{z+\rho^{-1}}
-
\frac{2z-(z^{q+1}+z^{N-q+1})}{z+\rho}
\right\}.
\end{equation}

We first compute the sum with denominator $z+\rho$.  
Since $z=\omega^j$,
\[
\frac{z^m}{z+\rho}=\frac{\omega^{jm}}{\rho+\omega^j}.
\]
We apply the root-of-unity identity with $m=1$, $m=q+1$, and $m=N-q+1$, reducing the exponents modulo $N$ when necessary.
Thus
\[
\sum_{j=0}^{N-1}\frac{z}{z+\rho}
=\frac{N}{\rho^N+1},
\]
\[
\sum_{j=0}^{N-1}\frac{z^{q+1}}{z+\rho}
=\frac{N(-1)^{q}\rho^{q}}{\rho^N+1},
\qquad
\sum_{j=0}^{N-1}\frac{z^{N-q+1}}{z+\rho}
=-\frac{N(-1)^{q}\rho^{N-q}}{\rho^N+1}.
\]

Therefore,
\begin{equation}\label{eq:block-rho-rev}
\sum_{j=0}^{N-1}
\frac{2z-(z^{q+1}+z^{N-q+1})}{z+\rho}
=\frac{N}{\rho^N+1}
\left\{
2-(-1)^{q}(\rho^{q}-\rho^{N-q})
\right\}.
\end{equation}

Similarly, for the sum with denominator $z+\rho^{-1}$, we obtain
\[
\sum_{j=0}^{N-1}\frac{z}{z+\rho^{-1}}
=\frac{N\rho^{N}}{\rho^N+1},
\]
\[
\sum_{j=0}^{N-1}\frac{z^{q+1}}{z+\rho^{-1}}
=\frac{N(-1)^{q}\rho^{N-q}}{\rho^N+1},
\qquad
\sum_{j=0}^{N-1}\frac{z^{N-q+1}}{z+\rho^{-1}}
=-\frac{N(-1)^{q}\rho^{q}}{\rho^N+1}.
\]
Hence
\begin{equation}\label{eq:block-rhoinv-rev}
\sum_{j=0}^{N-1}
\frac{2z-(z^{q+1}+z^{N-q+1})}{z+\rho^{-1}}
=
\frac{N}{\rho^N+1}
\left\{
2\rho^N-(-1)^{q}(\rho^{N-q}-\rho^{q})
\right\}.
\end{equation}

Substituting \eqref{eq:block-rho-rev} and \eqref{eq:block-rhoinv-rev} into \eqref{eq:Rsplit-rev}, we obtain
\begin{align*}
R^{(1)}(u,v)
&=\frac{1}{N\Delta}\left[
\frac{N}{\rho^N+1}
\left\{
2\rho^N-(-1)^{q}(\rho^{N-q}-\rho^{q})
\right\}
\right.\\
&\hspace{35mm}\left.
-
\frac{N}{\rho^N+1}
\left\{
2-(-1)^{q}(\rho^{q}-\rho^{N-q})
\right\}
\right] \\
&=\frac{2}{\Delta(\rho^N+1)}
\left\{
\rho^N-1+(-1)^{q}(\rho^{q}-\rho^{N-q})
\right\},
\end{align*}
which is the desired closed form.

Since $N$ is odd, we have $(-1)^{N-q}=-(-1)^{q}$, and thus
\[
\rho^N-1+(-1)^{N-q}\left(\rho^{N-q}-\rho^q\right)
=
\rho^N-1+(-1)^q\left(\rho^q-\rho^{N-q}\right).
\]
Hence the right-hand side is invariant under $q\mapsto N-q$.

Finally, we justify the positivity for $u\ne v$ directly from the spectral representation \eqref{eq:Rfourier-correct-rev}.
For $N\ge 5$, we have
\[
(N-2)+2\cos\!\left(\frac{2\pi j}{N}\right)\ge N-4>0
\quad (1\le j\le N-1),
\]
so all denominators are strictly positive.
Moreover, if $u\ne v$, then $q\not\equiv 0\pmod N$, and hence
\[
1-\cos\!\left(\frac{2\pi jq}{N}\right)\ge 0
\]
for all $j$, and this quantity is not identically zero as $j$ varies.
Therefore, each term in \eqref{eq:Rfourier-correct-rev} is nonnegative and at least one term is strictly positive.
Consequently,
\[
R^{(1)}(u,v)>0 \quad (u\ne v).
\]
\end{proof}

\begin{remark}
The essential point of the proof is that the denominator of the eigenvalues is a linear cosine expression
\[
(N-2)+2\cos\theta.
\]
After passing to the complex variable, this leads to the quadratic polynomial
$z^2+(N-2)z+1$ whose roots are $\rho$ and $\rho^{-1}$.  
This ``quadratic reduction'' is the source of the closed form.
\end{remark}

%%%%%%%%%%%%%%%%%%%%%%%%%%%%%%%%%%%%%%%%
\begin{theorem} \label{thm42}
Let $N\ge 5$ be odd. Then
\[
\tau(G_{N,1})
=\frac{1}{N}\frac{(\rho^{N}+1)^2}{\rho^{N-1}(\rho+1)^2}
\]
holds.
\end{theorem}

\begin{proof}
By Theorem \ref{thm21} and Lemma \ref{lem41},
\[
\tau(G_{N,1})=\frac{1}{N}\prod_{j=1}^{N-1}\lambda_j
=\frac{1}{N}\prod_{j=1}^{N-1}\left((N-2)+\omega^j+\omega^{-j}\right).
\]
Note that
\[
(N-2)+\omega^j+\omega^{-j}
=\frac{\omega^{2j}+(N-2)\omega^j+1}{\omega^j}
=\frac{(\omega^j+\rho)(\omega^j+\rho^{-1})}{\omega^j}.
\]
Hence
\[
\prod_{j=1}^{N-1}\lambda_j
=
\left(\prod_{j=1}^{N-1}(\omega^j+\rho)\right)
\left(\prod_{j=1}^{N-1}(\omega^j+\rho^{-1})\right)
\left(\prod_{j=1}^{N-1}\omega^{-j}\right).
\]
If $N$ is odd, then
\[
\prod_{j=1}^{N-1}\omega^{j}=\omega^{\frac{N(N-1)}{2}}=e^{\pi i (N-1)}=1,
\]
and thus $\prod_{j=1}^{N-1}\omega^{-j}=1$.
Moreover,
\[
\prod_{j=1}^{N-1}(\rho+\omega^j)=\frac{\rho^N+1}{\rho+1}.
\]
Similarly,
\[
\prod_{j=1}^{N-1}(\rho^{-1}+\omega^j)=\frac{\rho^{-N}+1}{\rho^{-1}+1}
=\frac{\rho(\rho^{-N}+1)}{\rho+1}.
\]
Multiplying these identities yields
\[
\prod_{j=1}^{N-1}\lambda_j
=\frac{\rho^N+1}{\rho+1}\cdot \frac{\rho(\rho^{-N}+1)}{\rho+1}
=\frac{(\rho^N+1)^2}{\rho^{N-1}(\rho+1)^2}.
\]
Finally, using $\tau(G_{N,1})=\frac{1}{N}\prod_{j=1}^{N-1}\lambda_j$ completes the proof.
\end{proof}

%%%%%%%%%%%%%%%%%%%%%%%%%%%%%%%%%%%%%%%%
\begin{corollary} \label{cor41}
For distinct vertices $u$ and $v$,
\[
F^{(1)}(u \mid v)
=\frac{2(\rho^N+1)}{\Delta N\rho^{N-1}(\rho+1)^2}
\left\{\rho^N-1+(-1)^{q}(\rho^{q}-\rho^{N-q})\right\}.
\]
\end{corollary}

\begin{proof}
This follows from Theorem \ref{thm22}, which states that
\[
F(u \mid v)=\tau(G)\,R(u,v).
\]
Substituting the explicit expressions for $\tau(G_{N,1})$ from Theorem \ref{thm42}
and $R^{(1)}(u,v)$ from Theorem \ref{thm41} yields the result.
\end{proof}

%%%%%%%%%%%%%%%%%%%%%%%%%%%%%%%%%%%%%%%%
\begin{corollary} \label{cor42}
The number of edges of $G_{N,1}$ is $m=\frac{N(N-3)}{2}$.  
Consequently, the expected hitting time satisfies
\[
H^{(1)}(u,v)
=\frac{N(N-3)}{\Delta(\rho^{N}+1)}
\left\{\rho^N-1+(-1)^{q}(\rho^{q}-\rho^{N-q})\right\}.
\]
\end{corollary}

\begin{proof}
Since $G_{N,1}$ is $(N-3)$-regular, we have
\[
\mathrm{vol}(G_{N,1})=N(N-3).
\]
By the result in Section \ref{sec02}, the symmetry $H(u,v)=H(v,u)$ holds,
and hence
\[
H(u,v)=\frac{\mathrm{vol}(G_{N,1})}{2}\,R(u,v).
\]
Substituting the explicit expression for $R^{(1)}(u,v)$ from Theorem \ref{thm41}
yields the desired formula.
\end{proof}

%%%%%%%%%%%%%%%%%%%%%%%%%%%%%%%%%%%%%%%%%%%%%%%%%%%%%%%%%%%%%%%%%%%%%%%%%%%%%%
\subsection{The complete graph with distance-$2$ edges removed} \label{sub42}
%%%%%%%%%%%%%%%%%%%%%%%%%%%%%%%%%%%%%%%%%%%%%%%%%%%%%%%%%%%%%%%%%%%%%%%%%%%%%%

We again assume that $N \ge 5$ is odd.  
Let $G_{N,2}$ be the circulant graph obtained from the complete graph $K_N$ by deleting all edges of distance $2$.

The case $r=2$ is a special instance of a more general phenomenon.
When $\gcd(r,N)=1$, multiplication by $r^{-1}$ defines an automorphism
of the cyclic group $\mathbb Z_N$, which induces a graph isomorphism
\[
G_{N,r} \cong G_{N,1}.
\]

In particular, since $N$ is odd, we have $\gcd(2,N)=1$, and hence
\[
G_{N,2} \cong G_{N,1}.
\]

Let $\varphi_2(x)=2^{-1}x \pmod N$ denote the isomorphism.
Then for $u,v\in\mathbb Z_N$,
\[
R_{G_{N,2}}(u,v)=R_{G_{N,1}}(\varphi_2(u),\varphi_2(v)),
\quad
\tau(G_{N,2})=\tau(G_{N,1}).
\]

%%%%%%%%%%%%%%%%%%%%%%%%%%%%%%%%%%%%%%%%
\begin{theorem} \label{thm43}
Let $N \ge 5$ be odd. 
Let $q = v-u \pmod N$, and let $s = 2^{-1} \pmod N$.
Define
\[
\delta_2(q)=\min\{sq \bmod N,\; N-(sq \bmod N)\}.
\]
Then
\[
R^{(2)}(u,v)=R^{(1)}_{\delta_2(q)}.
\]
Equivalently, using Theorem \ref{thm41}, this is written explicitly as
\[
R^{(2)}(u,v)
=
\frac{2}{\Delta(\rho^N+1)}
\left\{
\rho^N-1
+
(-1)^{\delta_2(q)}
\left(
\rho^{\delta_2(q)}
-
\rho^{N-\delta_2(q)}
\right)
\right\}.
\]
\end{theorem}

\begin{proof}
By the graph isomorphism $G_{N,2} \cong G_{N,1}$ induced by
$\varphi_2(x)=2^{-1}x$, we have
\[
R^{(2)}(u,v)=R^{(1)}(\varphi_2(u),\varphi_2(v)).
\]

Since $\varphi_2(v)-\varphi_2(u) \equiv s(v-u) \equiv sq \pmod N$,
the corresponding distance in $G_{N,1}$ is given by
\[
\delta_2(q)=\min\{sq \bmod N,\; N-(sq \bmod N)\}.
\]

Therefore,
\[
R^{(2)}(u,v)=R^{(1)}_{\delta_2(q)}.
\]
The explicit formula follows by substituting $\delta_2(q)$ into
Theorem \ref{thm41}.
\end{proof}
%%%%%%%%%%%%%%%%%%%%%%%%%%%%%%%%%%%%%%%%
\begin{theorem} \label{thm44}
Let $N \ge 5$ be odd. Then
\[
\tau(G_{N,2})=\tau(G_{N,1}).
\]
\end{theorem}

\begin{proof}
This follows immediately from the graph isomorphism
$G_{N,2} \cong G_{N,1}$.
\end{proof}

%%%%%%%%%%%%%%%%%%%%%%%%%%%%%%%%%%%%%%%%
\begin{corollary} \label{cor43}
For distinct vertices $u$ and $v$,
\[
F^{(2)}(u \mid v)=\tau(G_{N,2})\,R^{(2)}(u,v).
\]

In particular,
\[
F^{(2)}(u \mid v)=\tau(G_{N,1})\,R^{(1)}_{\delta_2(q)}.
\]

Using Theorems \ref{thm41} and \ref{thm42}, this admits the explicit form
\[
F^{(2)}(u \mid v)
=
\frac{2(\rho^N+1)}{\Delta N\rho^{N-1}(\rho+1)^2}
\left\{
\rho^N-1
+
(-1)^{\delta_2(q)}
\left(
\rho^{\delta_2(q)}-\rho^{N-\delta_2(q)}
\right)
\right\}.
\]
\end{corollary}

%%%%%%%%%%%%%%%%%%%%%%%%%%%%%%%%%%%%%%%%
\begin{corollary} \label{cor44}
The graph $G_{N,2}$ is $(N-3)$-regular and has $m=\frac{N(N-3)}{2}$ edges.  

Consequently,
\[
H^{(2)}(u,v)=\frac{\mathrm{vol}(G_{N,2})}{2}\,R^{(2)}(u,v).
\]

In particular,
\[
H^{(2)}(u,v)=\frac{N(N-3)}{2}\,R^{(1)}_{\delta_2(q)}.
\]

Hence, using Theorem \ref{thm41}, we obtain the explicit formula
\[
H^{(2)}(u,v)
=
\frac{N(N-3)}{\Delta(\rho^N+1)}
\left\{
\rho^N-1
+
(-1)^{\delta_2(q)}
\left(
\rho^{\delta_2(q)}-\rho^{N-\delta_2(q)}
\right)
\right\}.
\]
\end{corollary}

%%%%%%%%%%%%%%%%%%%%%%%%%%%%%%%%%%%%%%%%
\begin{remark}
The reduction of the case $r=2$ to $r=1$ is not merely a consequence
of a permutation of eigenvalues, but follows from a structural
isomorphism of graphs. This explains why all quantities for $r=2$
admit explicit closed forms in terms of those for $r=1$.
\end{remark}

%%%%%%%%%%%%%%%%%%%%%%%%%%%%%%%%%%%%%%%%%%%%%%%%%%%%%%%%%%%%%%%%%%%%%%%%%%%%%%
\section{Asymptotic and structural consequences} \label{sec05}

In this section, we analyze in detail the asymptotic structure and combinatorial consequences derived from the closed forms obtained in Section \ref{sec04}.  
In particular, we evaluate explicitly the limiting behavior of the resistance distance, the Kirchhoff index, and the exponential growth rate of the number of spanning trees.  
Throughout this section, we consider asymptotics as $N\to\infty$ through odd integers,
and all statements involving a fixed distance mean that the parameter is independent of $N$.

By Theorem \ref{thm41},
\[
R^{(1)}(u,v)
=\frac{2}{\Delta(\rho^N+1)}\left\{\rho^N-1+(-1)^{q}(\rho^{q}-\rho^{N-q})\right\},
\]
where
\[
\Delta=\sqrt{N(N-4)},
\qquad
\rho=\frac{N-2+\Delta}{2}.
\]

%%%%%%%%%%%%%%%%%%%%%%%%%%%%%%%%%%%%%%%%
%%%%%%%%%%%%%%%%%%%%%%%%%%%%%%%%%%%%%%%%
\subsection{Asymptotics of $\rho$ and $\Delta$}

We first obtain an asymptotic expansion of $\rho$.  Since
\[
\Delta=\sqrt{N^2-4N}
=
N\sqrt{1-\frac{4}{N}},
\]
applying the binomial expansion
\[
\sqrt{1+x}
=1+\frac{x}{2}-\frac{x^2}{8}+O(x^3)
\]
with $x=-4/N$ yields
\[
\Delta
=
N\left(1-\frac{2}{N}-\frac{2}{N^2}+O(N^{-3})\right).
\]
Here $O(N^{-3})$ denotes terms of order at most $N^{-3}$ as $N\to\infty$.

Consequently,
\[
\rho
=
\frac{N-2+N-2-\frac{2}{N}+O(N^{-2})}{2}
=
N-2-\frac{1}{N}+O(N^{-2}).
\]
In particular,
\[
\rho \sim N-2,
\qquad
\Delta \sim N.
\]
\subsection{Limit of $R^{(1)}(u,v)$ for fixed $q$}

Fix an integer $q \ge 1$ independent of $N$, and let $N\to\infty$.
Then
\[
\rho^{N-q}
=\rho^N \rho^{-q}
=\rho^N\,O(N^{-q}),
\]
and hence
\[
\rho^N-1+(-1)^{q}(\rho^{q}-\rho^{N-q})
=\rho^N\left(1-(-1)^q\rho^{-q}\right)+(-1)^q\rho^q-1.
\]

Since $\rho\to\infty$ as $N\to\infty$, we have
\[
1-(-1)^q\rho^{-q} \to 1.
\]
Thus, the numerator is asymptotic to $\rho^N$ for all fixed $q\ge 1$.
Since $\rho^N+1\sim \rho^N$, we obtain
\[
R^{(1)}(u,v)
\sim
\frac{2}{\Delta}
\sim
\frac{2}{N}.
\]

%%%%%%%%%%%%%%%%%%%%%%%%%%%%%%
%%%%%%%%%%% Theorem 5.1 %%%%%%%
%%%%%%%%%%%%%%%%%%%%%%%%%%%%%%
\begin{theorem}
For each fixed integer $q \ge 1$ independent of $N$,
\[
R^{(1)}(u,v) \sim \frac{2}{N} \qquad (N\to\infty).
\]
\end{theorem}

This is the same asymptotic behavior as the resistance distance in the complete graph $K_N$,
\[
R_{K_N}(u,v)=\frac{2}{N}.
\]

%%%%%%%%%%%%%%%%%%%%%%%%%%%%%%
%%%%%%%%%%%%%%%%%%%%%%%%%%%%%%
\subsection{Kirchhoff index}

We use the standard spectral identity
\[
Kf(G)=N\sum_{j=1}^{N-1}\frac{1}{\lambda_j}.
\]

Applying this to $G_{N,1}$, we obtain
\[
Kf(G_{N,1})
=
N\sum_{j=1}^{N-1}\frac{1}{N-2+2\cos\left(\frac{2\pi j}{N}\right)}.
\]

Alternatively, using pairwise distances,
\[
Kf(G)=\frac{1}{2}\sum_{u,v}R(u,v),
\]
and for circulant graphs,
\[
Kf(G_{N,1})
=
\frac{N}{2}\sum_{q=1}^{N-1}R^{(1)}(q),
\]
where $q=v-u \pmod N$ is the oriented residue.
The factor $N/2$ accounts for double counting of unordered pairs.

Substituting Theorem \ref{thm41}, we obtain
\[
Kf(G_{N,1})
=\frac{N}{\Delta(\rho^N+1)}\sum_{q=1}^{N-1}\left\{\rho^N-1+(-1)^{q}(\rho^{q}-\rho^{N-q})\right\}.
\]

Let
\[
S(N,\rho)
:=\sum_{q=1}^{N-1}\left\{\rho^N-1+(-1)^{q}(\rho^{q}-\rho^{N-q})\right\}.
\]

Then
\[
S(N,\rho)
=
\sum_{q=1}^{N-1}(\rho^N-1)
+\sum_{q=1}^{N-1}(-1)^{q}\rho^{q}
-\sum_{q=1}^{N-1}(-1)^{q}\rho^{N-q}.
\]

First,
\[
\sum_{q=1}^{N-1}(\rho^N-1)=(N-1)(\rho^N-1).
\]

Next, by the geometric series,
\[
\sum_{q=1}^{N-1}(-1)^{q}\rho^{q}
=
\sum_{q=1}^{N-1}(-\rho)^{q}
=\frac{(-\rho)\{1-(-\rho)^{N-1}\}}{1-(-\rho)}
=-\frac{\rho\{1-(-\rho)^{N-1}\}}{1+\rho}.
\]

Since $N$ is odd, $N-1$ is even, and hence $(-\rho)^{N-1}=\rho^{N-1}$. Therefore,
\begin{equation}\label{eq:sum1}
\sum_{q=1}^{N-1}(-1)^{q}\rho^{q}
=
-\frac{\rho(1-\rho^{N-1})}{1+\rho}
=
\frac{\rho(\rho^{N-1}-1)}{1+\rho}
=
\frac{\rho^N-\rho}{1+\rho}.
\end{equation}

Similarly,
\[
\sum_{q=1}^{N-1}(-1)^{q}\rho^{N-q}
=
\rho^N\sum_{q=1}^{N-1}(-1)^{q}\rho^{-q}
=
\rho^N\sum_{q=1}^{N-1}(-\rho^{-1})^{q}.
\]

Again by the geometric series,
\[
\sum_{q=1}^{N-1}(-\rho^{-1})^q
=\frac{(-\rho^{-1})\{1-(-\rho^{-1})^{N-1}\}}{1+\rho^{-1}}.
\]

Since $N-1$ is even, $(-\rho^{-1})^{N-1}=\rho^{-(N-1)}$, and using $1+\rho^{-1}=(\rho+1)/\rho$ we obtain
\[
\sum_{q=1}^{N-1}(-\rho^{-1})^{q}
=
\frac{(-\rho^{-1})(1-\rho^{-(N-1)})}{(\rho+1)/\rho}
=-\frac{1-\rho^{-(N-1)}}{\rho+1}.
\]

Thus,
\begin{equation}\label{eq:sum2}
\sum_{q=1}^{N-1}(-1)^{q}\rho^{N-q}
=
\rho^N\left(-\frac{1-\rho^{-(N-1)}}{\rho+1}\right)
=-\frac{\rho^N-\rho}{\rho+1}.
\end{equation}

Using \eqref{eq:sum1} and \eqref{eq:sum2}, we have
\[
\sum_{q=1}^{N-1}(-1)^{q}(\rho^{q}-\rho^{N-q})
=
\frac{\rho^N-\rho}{\rho+1}
-\left(-\frac{\rho^N-\rho}{\rho+1}\right)
=
\frac{2(\rho^N-\rho)}{\rho+1}.
\]

Hence
\[
S(N,\rho)
=
(N-1)(\rho^N-1)+\frac{2(\rho^N-\rho)}{\rho+1}.
\]

Substituting this into $Kf(G_{N,1})=\dfrac{N}{\Delta(\rho^N+1)}S(N,\rho)$ yields the closed form
\begin{equation}\label{eq:Kf-closed}
Kf(G_{N,1})
=
\frac{N}{\Delta(\rho^N+1)}
\left\{(N-1)(\rho^N-1)+\frac{2(\rho^N-\rho)}{\rho+1}\right\}.
\end{equation}

We now derive the asymptotics of $Kf(G_{N,1})$.  
Since $\rho>1$, the term $\rho^{-N}$ is exponentially small, and hence
\[
\frac{\rho^N-1}{\rho^N+1}=1+O(\rho^{-N}),
\qquad
\frac{\rho^N-\rho}{\rho^N+1}=1+O(\rho^{-1}).
\]

Therefore, from \eqref{eq:Kf-closed},
\[
Kf(G_{N,1})
=\frac{N}{\Delta}\left((N-1)+\frac{2}{\rho+1}\right)+O\!\left(\frac{N}{\Delta}\rho^{-N}\right).
\]

Since $\rho\sim N$ and $\Delta\sim N$, we have
\[
\frac{N}{\Delta}\cdot\frac{2}{\rho+1}=O(N^{-1}),
\]
and hence
\[
Kf(G_{N,1})
=
\frac{N(N-1)}{\Delta}+O(N^{-1}).
\]

Thus we obtain the following.

%%%%%%%%%%%%%%%%%%%%%%%%%%%%%%
%%%%%%%%%%% Theorem 5.2 %%%%%%%
%%%%%%%%%%%%%%%%%%%%%%%%%%%%%%
\begin{theorem}
\[
Kf(G_{N,1})
\sim
\frac{N^2}{\Delta}
\sim
\frac{N^2}{N}
=N
\qquad (N\to\infty).
\]
\end{theorem}

For the complete graph, $Kf(K_N)=N-1$, and hence the Kirchhoff index remains of the same order.

%%%%%%%%%%%%%%%%%%%%%%%%%%%%%%
%%%%%%%%%%%%%%%%%%%%%%%%%%%%%%
\subsection{Asymptotics of the number of spanning trees}

By Theorem \ref{thm42},
\[
\tau(G_{N,1})
=\frac{1}{N}\frac{(\rho^N+1)^2}{\rho^{N-1}(\rho+1)^2}.
\]
Using $\rho\sim N-2$, we obtain
\[
\tau(G_{N,1})
\sim
\frac{\rho^{2N}}{N\,\rho^{N-1}\rho^2}
=
\frac{\rho^{N-1}}{N},
\]
and hence
\[
\tau(G_{N,1})
\sim
\frac{(N-2)^{N-1}}{N}.
\]
On the other hand, by Cayley's formula,
\[
\tau(K_N)=N^{N-2}.
\]
Taking the ratio gives
\[
\frac{\tau(G_{N,1})}{\tau(K_N)}
\sim
\frac{(N-2)^{N-1}}{N^{N-1}}
=
\left(1-\frac{2}{N}\right)^{N-1}.
\]
Therefore we obtain:

%%%%%%%%%%%%%%%%%%%%%%%%%%%%%%
%%%%%%%%%%% Theorem 5.3 %%%%%%%
%%%%%%%%%%%%%%%%%%%%%%%%%%%%%%
\begin{theorem}
\[
\lim_{N\to\infty}
\frac{\tau(G_{N,1})}{\tau(K_N)}
=
e^{-2}.
\]
\end{theorem}

In other words, deleting all distance-$1$ edges decreases the number of spanning trees by the exponential factor $e^{-2}$.

\medskip
In summary, the distance-$1$ deletion graph $G_{N,1}$ has the same local order of resistance as the complete graph.  
At the same time, the number of spanning trees decays exponentially by a factor $e^{-2}$, while the Kirchhoff index grows linearly in $N$.

%%%%%%%%%%%%%%%%%%%%%%%%%%%%%%%%%%%%%%%%%%%%%%%%%%%%
\section{Conclusion} \label{sec06}
%%%%%%%%%%%%%%%%%%%%%%%%%%%%%%%%%%%%%%%%%%%%%%%%%%%%

In this paper, we studied a family of circulant graphs obtained from the complete graph $K_N$ by deleting all edges belonging to a prescribed cyclic distance class. 
We derived, in a unified manner, formulas for the effective resistance, the expected hitting time, the number of spanning trees, and the number of two-component spanning forests, based on the Laplacian eigenvalues. 
For general distance-class deletions, these quantities admit natural spectral representations, but one should not expect readable closed forms in general.

In particular, we clarified that for a single deleted distance class $r$ with $\gcd(r,N)=1$, the graph $G_{N,r}$ is isomorphic to $G_{N,1}$. 
Thus, all spectral and combinatorial quantities in this setting reduce to the case $r=1$.

On the other hand, when the deleted distance class is $r = 1$ (and hence also for all $r$ with $\gcd(r,N)=1$), and the number of vertices $N$ is assumed to be odd, we showed that remarkably simple closed forms hold. 
For the distance-1 deletion, the eigenvalue denominator becomes a linear cosine expression, which reduces to a quadratic polynomial under the complex-variable substitution. 
This yields explicit exponential-type formulas for the resistance distance and for the spanning tree count.

We emphasize that the resistance formula itself is closely related to earlier work of Chair \cite{Ref09}, where exact expressions for the complete graph minus a Hamiltonian cycle were derived. 
The contribution of the present paper lies in providing a unified spectral derivation and in extending the analysis to include explicit formulas for spanning trees, two-component spanning forests, and hitting times, as well as their asymptotic behavior.

Furthermore, the case $r=2$ is not a separate phenomenon but a special instance of this general isomorphism. 
Indeed, the map $x \mapsto 2^{-1}x \pmod N$ induces a graph isomorphism $G_{N,2} \cong G_{N,1}$ when $N$ is odd, and therefore both the effective resistance and the spanning tree count reduce completely to those of the case $r=1$. 
This shows that the coincidence of spectra is a consequence of a structural equivalence rather than a mere permutation of eigenvalues.

We also performed an asymptotic analysis. 
For each fixed distance $h=h(u,v)\ge 1$ independent of $N$, the resistance distance has the same order $2/N$ as in $K_N$, whereas the number of spanning trees decays exponentially by the factor $e^{-2}$. 
In addition, we showed that the Kirchhoff index remains of linear order in $N$.

These results demonstrate that the simple operation of deleting a distance class produces precise and quantitative changes in the spectral structure and in combinatorial invariants of highly symmetric graphs.

At the same time, we note that for more general distance sets $S$, especially when $\gcd(r,N)>1$ or when multiple distance classes are removed, the structure may differ substantially, and the present analysis does not yield comparable closed forms. 
Thus, the general theory remains at the level of spectral representations in such cases.

Possible directions for future work include the analysis of the eigenvalue structure when multiple distance classes are deleted simultaneously, as well as a systematic understanding of the symmetry-breaking phenomena that appear for even $N$. 
It is also an interesting problem to investigate whether similar closed forms can be obtained for other highly symmetric graph families.

%%%%%%%%%%%%%%%%%%%%%%%%%%%%%%
%%%%%%%%%%%%%%%%%%%%%%%%%%%%%%
\section*{Acknowledgments}
The author would like to thank the anonymous referee for their careful reading and valuable comments, which significantly improved the quality of this paper.

%%%%%%%%%%%%%%%%%%%%%%%%%%%%%%
%%%%%%%%%%%%%%%%%%%%%%%%%%%%%%

\end{document}